\documentclass[a4paper,11pt]{article}
\usepackage{hyperref}
\usepackage{anysize}
\usepackage{url}
\usepackage{paralist}
\usepackage{amsthm}
\usepackage{amsmath}
\usepackage{amssymb}
\usepackage{bbm}

\newcommand{\Hp}{H^\prime}
\newcommand{\Hpp}{H^{\prime\prime}}
\newcommand\tchi{\tilde{\chi}}
\newcommand\w{\mathbf{\omega}}
\newcommand\wb{\omega}
\newcommand\FF{\mathcal{P}}
\newcommand\R{\mathbb{R}}
\newcommand\conv{\mathsf{conv}\,}
\renewcommand\min{\mathsf{min}\,}
\newcommand\tprism{\mathsf{tprism}\,}
\newcommand\bip{\mathsf{bip}\,}
\newcommand\prism{\mathsf{prism}\,}
\newcommand\pyr{\mathsf{pyr}\,}
\newcommand\A{{\bf A}}
\newcommand\B{{\bf B}}
\newcommand\C{{\bf C}}
\newcommand\x{\mathbf{x}}
\newcommand\p{\mathbf{p}}

\newcommand\ZO{\{0,1\}}
\newcommand\HS{\tilde{\Delta}}
\renewcommand\dim{\mathsf{dim}\,}
\newcommand\rk{\mathsf{rank}\,}
\newcommand\1{\mathbbmss{1}}
\newcommand\PM{\{+1,-1\}}
\newcommand\Z{\mathbbmss{Z}}
\renewcommand\a{\alpha}
\newcommand\dual{\triangle}
\newcommand\tor{\mathsf{tor}}
\newcommand\ic[1]{\mathsf{Ind}(#1)}
\newcommand\sym{\mathsf{sym}}

\newcommand\Sphi{\Phi^\sym}

\newcommand\bV{\overline{V}}
\newcommand\bE{\overline{E}}

\newtheorem{thm}{Theorem}[section]
\newtheorem{prop}[thm]{Proposition}
\newtheorem{lem}[thm]{Lemma}
\newtheorem{cor}[thm]{Corollary}

\theoremstyle{definition}
\newtheorem{dfn}[thm]{Definition}
\newtheorem{example}[thm]{Example}
\theoremstyle{remark}

\begin{document}
\title{On Kalai's conjectures\\concerning centrally symmetric polytopes} 
\author{\Large Raman Sanyal \qquad Axel Werner \qquad 
         G\"{u}nter M.  Ziegler\\[3mm]
    Institute of Mathematics, MA 6-2\\
    TU Berlin\\
    D-10623 Berlin, Germany\\
    \tt\{sanyal,awerner,ziegler\}@math.tu-berlin.de
} 

\date{September 30, 2007}

\maketitle

\begin{abstract}\noindent
  In 1989 Kalai stated the three conjectures \A, \B, \C\ of increasing
  strength concerning face numbers of centrally symmetric convex polytopes.
  The weakest conjecture, \A, became known as the ``$3^d$-conjecture''.  It is
  well-known that the three conjectures hold in dimensions $d \le 3$. We show
  that in dimension $4$ only conjectures \A\ and \B\ are valid, while
  conjecture \C\ fails. Furthermore, we show that both conjectures \B\ and \C\
  fail in all dimensions $d \ge 5$.
\end{abstract}

\section{Introduction}\label{sec:intro}

A convex $d$-polytope $P$ is \emph{centrally symmetric}, or \emph{cs} for
short, if $P = -P$. Concerning face numbers, this implies that for $0 \le i
\le  d-1$ the number of $i$-faces $f_i(P)$ is even and, since $P$ is
full-dimensional, that $\min \{f_0(P),f_{d-1}(P)\} \ge 2d$.  Beyond this, only
very little is known for the general case. That is to say, the extra
(structural) information of a central symmetry yields no substantial
additional constraints for the face numbers on the restricted class of
polytopes.

Not uncommon to the $f$-vector business, the knowledge about face numbers is
concentrated on the class of centrally symmetric \emph{simplicial}, or dually
\emph{simple}, polytopes. In 1982, B\'{a}r\'{a}ny and Lov\'{a}sz \cite{bar82}
proved a lower bound on the number of vertices of simple cs polytopes with
prescribed number of facets, using a generalization of the Borsuk--Ulam
theorem. Moreover, they conjectured lower bounds for all face numbers of this
class of polytopes with respect to the number of facets. In 1987 Stanley
\cite{stan87} proved a  conjecture of Bj\"{o}rner concerning the $h$-vectors
of simplicial cs polytopes that implies the one by B\'{a}r\'{a}ny and
Lov\'{a}sz. The proof uses Stanley-Reisner rings and toric varieties plus a
pinch of representation theory. The result of Stanley \cite{stan87} for cs
polytopes was reproved in a more geometric setting by Novik \cite{nov99} 
by using ``symmetric flips'' in McMullen's weight algebra \cite{mcm96}. 
For general polytopes, lower bounds on the \emph{toric} $h$-vector were
recently obtained by A'Campo-Neuen \cite{aca06} by using combinatorial
intersection cohomology.  Unfortunately, the toric $h$-vector contains only
limited information about the face numbers of general (cs) polytopes and thus
the applicability of the result is limited (see Section~\ref{sec:rigid}).

In \cite{kal89}, Kalai stated three conjectures about the face numbers of
general cs polytopes. Let $P$ be a (cs) $d$-polytope with $f$-vector $f(P) =
(f_0,f_1,\dots,f_{d-1})$. Define the function $s(P)$ by
\[
    s(P) := 1 + \sum_{i=0}^{d-1}f_i(P) = f_P(1)
\]
where $f_P(t) := f_{d-1}(P) + f_{d-2}(P)t + \cdots + f_0(P)t^{d-1}+ t^d$ 
is the $f$-polynomial.
Thus, $s(P)$ measures the total number of non-empty faces of $P$. Here is
Kalai's first conjecture from \cite{kal89}, the ``$3^d$-conjecture''.

{\bf Conjecture A.} Every centrally-symmetric $d$-polytope has at least $3^d$
non-empty faces, i.e.\ $s(P) \ge 3^d$.

Is is easy to see that the bound is attained for the $d$-dimensional cube
$C_d$ and for its dual, the $d$-dimensional crosspolytope $C_d^\dual$. It
takes a moment's thought to see that in dimensions $d \ge 4$ these are not the
only polytopes with $3^d$ non-empty faces. An important class that attains the
bound is the class of \emph{Hanner polytopes} \cite{han56}.  These are defined
recursively: As a start, every cs $1$-dimensional polytope is a Hanner
polytope. For dimensions $d \ge 2$, a $d$-polytope $H$ is a Hanner polytope if
it is the direct sum or the direct product of two (lower dimensional) Hanner
polytopes $\Hp$ and $\Hpp$.

The number of Hanner polytopes grows exponentially in the dimension~$d$, with
a Catalan-type recursion. It is given by the number of two-terminal networks
with $d$ edges, $n(d)=1, 1, 2, 4, 8, 18, 40, 94, 224, 548, 1356,\dots$, for
$d=1,2,\dots$, as counted by Moon \cite{Moon87}; see also
\cite{sloane:_number}.%

{\bf Conjecture B.} For every centrally-symmetric $d$-polytope $P$ there is a
$d$-dimensional Hanner polytope $H$ such that $f_i(P) \ge f_i(H)$ for all $i =
0, \dots, d-1$.

For a $d$-polytope $P$ and $S = \{i_1, i_2, \dots, i_k \} \subseteq [d] = \{0,
1, \dots, d-1\}$ let $f_S(P) \in \Z^{2^{[d]}}$ be the number of chains of
faces $F_1 \subset F_2 \subset \cdots \subset F_k \subset P$ with $\dim F_j =
i_j$ for all $j = 1,\dots,k$. Identifying $\R^{2^{[d]}}$ with its dual space
via the standard inner product, we write $\a(P) := \sum_S \a_S f_S(P)$ for
$(\a_S)_{S\subseteq[d]} \in \R^{2^{[d]}}$. The set
\[
    \FF_d = \bigl\{ (\a_S)_{S\subseteq[d]} \in \R^{2^{[d]}} :  
        \a(P) = \sum_S \a_S f_S(P) \ge 0 \text{ for all $d$-polytopes } P
    \bigr\}
\]
is the polar to the set of flag-vectors of $d$-polytopes, that is, the cone of
all linear functionals that are non-negative on all flag-vectors of (not
necessarily cs) $d$-polytopes.

{\bf Conjecture C.} For every centrally-symmetric $d$-polytope $P$ there is a
$d$-dimensional Hanner polytope $H$ such that $\a(P) \ge \a(H)$ for all $\a
\in \FF_d$.

It is easy to see that \C\ $\Rightarrow$ \B\ $\Rightarrow$ \A: Define $\a^i(P)
:= f_i(P)$, then $\a^i \in \FF_d$ and the validity of \C\ on the functionals
$\a^i$ implies \B; the remaining implication follows since $s(P)$ is a
non-negative combination of the $f_i(P)$.

In this paper we investigate the validity of these three conjectures in
various dimensions. Our main results are as follows.

\begin{thm}\label{thm:dim4}
    The conjectures \A\ and \B\ hold for centrally symmetric polytopes of
    dimension $d \le 4$.
\end{thm}

\begin{thm}\label{thm:Cdim4}
    Conjecture \C\ is false in dimension $d = 4$.
\end{thm}

\begin{thm}\label{thm:dim5}
    For all $d \ge 5$
    both conjectures \B\ and \C\ fail.
\end{thm}

The paper is organized as follows. In Section \ref{sec:AB_dim4} we establish a
lower bound on the flag-vector functional $g^\tor_2$ on the class of cs
$4$-polytopes. Together with some combinatorial and geometric reasoning this
leads to a proof of Theorem \ref{thm:dim4}. In Section \ref{sec:C_dim4}, we
exhibit a centrally symmetric $4$-polytope and a flag vector functional that
disprove conjecture~\C. In Section \ref{sec:hypsimp} we consider centrally
symmetric hypersimplices in odd dimensions; combined with basic properties of
Hanner polytopes, this gives a proof of Theorem \ref{thm:dim5}.  We close with
two further interesting examples of centrally symmetric polytopes in
Section~\ref{sec:examples}.

{\bf Acknowledgements.} 
We are grateful to Gil Kalai for his inspiring conjectures, and for pointing
out the connection to symmetric stresses for Theorem~\ref{thm:g2_symm}.

\section{Conjectures \A\ and \B\ in dimensions $\mathbf{d \le 4}$}
\label{sec:AB_dim4}

In this section we prove Theorem \ref{thm:dim4}, that is,  the conjectures \A\
and \B\ for polytopes in dimensions $d \le 4$.  The work of Stanley
\cite{stan87} implies \A\ and \B\ for simplicial and thus also for simple
polytopes.  Furthermore, if $f_0(P) = 2d$, then $P$ is linearly isomorphic to
a crosspolytope.  \emph{Therefore, we assume throughout this section that all
cs $d$-polytopes $P$ are neither simple nor simplicial, and that $f_{d-1}(P)
\ge f_0(P) \ge 2d+2$.} 

The main work will be in dimension $4$. The claims for dimensions one, two,
and three are vacuous, clear, and easy to prove, in that order. In particular,
the case $d = 3$ can be obtained from an easy $f$-vector calculation. But, to
get in the right mood, let us sketch a geometric argument.  Let $P$ be a cs
$3$-polytope.  Since $P$ is not simplicial, $P$ has a non-triangle facet.
Let $F$ be a facet of $P$ with $f_0(F) \ge 4$ vertices.  Let $F_0 = P \cap H$
with $H$ being the hyperplane parallel to the affine hulls of $F$ and of $-F$
that contains the origin. Now, $F_0$ is a cs $2$-polytope and it is clear that
every face $G$ of $P$ that has a nontrivial intersection with $H$ is neither a
face of $F$ nor of $-F$. We get
\[
    s(P)\ \ \ge\ \  s(F) + s(F_0) + s(-F) \ge 3 \cdot 3^2.
\]
This type of argument fails in dimensions $d \ge 4$.  Applying small
(symmetric) perturbations to the vertices of a prism over an octahedron yields
a cs $4$-polytope with the following two types of facets: prisms over a
triangle and square pyramids. Every such facet has less than $3^3$ faces,
which shows that less than a third of the alleged $81$ faces are concentrated
in any facet.

Let's come back to dimension~$4$.  The proof of the conjectures \A\ and \B\
splits into a combinatorial part (\emph{$f$-vector yoga}) and a geometric
argument. We partition the class of cs $4$-polytopes into \emph{large} and
(few) \emph{small} polytopes, where ``large'' means that
\begin{equation}\label{eqn:24}
    f_0(P) + f_3(P)\ \ \ge\ \ 24.
\end{equation}

We will reconsider an argument of Kalai \cite{kal87} that proves a lower bound
theorem for polytopes and, in combination with flag-vector identities, leads
to a tight flag-vector inequality for cs $4$-polytopes. With this new tool, we
prove that (\ref{eqn:24}) implies conjectures \A\ and \B\ for dimension~$4$.%

We show that the \emph{small} cs $4$-polytopes, i.e.\ those not satisfying
(\ref{eqn:24}), are \emph{twisted prisms}, to be introduced in Section
\ref{sec:twist}, over $3$-polytopes. We then establish basic properties of
twisted prisms that imply the validity of conjectures \A\ and~\B\ for small cs
$4$-polytopes.

\subsection{Rigidity with symmetry and flag-vector inequalities}
\label{sec:rigid}

For a general simplicial $d$-polytope $P$ the \emph{$h$-vector} $h(P)$ is the
ordered collection of the coefficients of the polynomial $h_P(t) := f_P(t-1)$,
the \emph{$h$-polynomial} of $P$. Clearly, $h_P(t)$ encodes the same
information as the $f$-polynomial, but additionally $h_P(t)$ is a unimodal,
palindromic polynomial with non-negative, integral coefficients (see e.g.\
\cite[Sect.~8.3]{zie95}). This gives more insight in the nature of face
numbers of simplicial polytopes and, in a compressed form, this numerical
information is carried by its \emph{$g$-vector} $g(P)$ with $g_i(P) = h_i(P) -
h_{i-1}(P)$ for $i = 1, \dots, \lfloor \frac{d}{2} \rfloor$. There are various
interpretations for the $h$- and $g$-numbers and, via the $g$-Theorem, they
carry a complete characterization of the $f$-vectors of simplicial
$d$-polytopes.

For general $d$-polytopes a much weaker invariant is given by the
\emph{generalized} or \emph{toric} $h$-vector $h^\tor(P)$ introduced by
Stanley~\cite{stan87-2}.  In contrast to the ordinary $h$-vector, the toric
$h$-numbers $h_i^\tor(P)$ are not determined by the $f$-vector: They are
linear combinations of the face numbers and of other entries of the
flag-vector of $P$. For example, 
\[
    g_2^\tor \ =\ h^\tor_2 - h^\tor_1\ \ =\ \
    f_1 + f_{02} - 3f_2 - df_0 + \tbinom{d+1}{2}.
\] 
The corresponding toric $h$-polynomial shares the same properties as its
simplicial relative but, unfortunately, carries quite incomplete information
about the $f$-vector. 

For example, in the case of $P$ being a \emph{quasi-simplicial} polytope,
i.e.\ if every facet of $P$ is simplicial, the toric $h$-vector depends only
on the $f$-numbers $f_i(P)$ for $0 \le i \le \lfloor \frac{d}{2} \rfloor$ and,
therefore, does not carry enough information to determine a lower bound on
$s(P)$ for $d\ge 5$. However, the information gained in dimension $4$ will be
a major step in the direction of a proof of Theorem \ref{thm:dim4}. To be more
precise, for the class of centrally symmetric $d$-polytopes there is a
refinement of the flag-vector inequality $g^\tor_2 = h^\tor_2 - h^\tor_1 \ge
0$.

\begin{thm}\label{thm:g2_symm}
  Let $P$ be a centrally symmetric $d$-polytope. Then 
  \[
        g_2^\tor (P) \ \ =\ \ f_1(P) + f_{02}(P) - 3f_2(P) - df_0(P) +
        \tbinom{d+1}{2} \ \ \ge\ \ \tbinom{d}{2}-d.
  \]
\end{thm}

With Euler's equation and the Generalized Dehn-Sommerville equations \cite{bb85}
it is routine to derive the following inequality for the class of cs $4$-polytopes.

\begin{cor} \label{cor:g2_symm}
    If $P$ is a centrally symmetric $4$-polytope, then
    \begin{equation}
        f_{03}(P)\ \ \ge\ \ 3f_0(P) + 3f_3(P) - 8.
        \label{eqn:g2}
    \end{equation}
\end{cor}

We will prove Theorem \ref{thm:g2_symm} using the theory of
\emph{infinitesimally rigid frameworks}.  For information about rigidity
beyond our needs we refer the reader to Roth~\cite{rot81} for a very readable
introduction and to Whiteley~\cite{whi84} and Kalai~\cite{kal89} for rigidity
in connection with polytopes.

Let $d \ge 1$ and let $G = (V,E)$ be an abstract simple undirected graph. The
\emph{edge function} associated to $G$ and $d$ is the map
\[
\begin{aligned}
    \Phi: (\R^d)^V &\rightarrow \R^E\\
                (p_v)_{v\in V} &\mapsto \left( \|p_u - p_v\|^2\right)_{uv \in
                E},
\end{aligned}
\]
which measures the (squared) lengths of the edges of $G$ for any choice of
coordinates $\p = (p_v)_{v\in V} \in (\R^d)^V$.  The pair $(G,\p)$ is called a
\emph{framework} in $\R^d$ and the points of $\Phi_\p := \Phi^{-1}(\Phi(\p))$
give the possible frameworks in $\R^d$ with constant edge lengths $\Phi(\p)$.

Let $v = |V| \ge d+1$ and let $\p$ be a generic embedding. Then the set
$\Phi_\p \subset (\R^d)^V$ is a smooth submanifold on which the group of
\emph{Euclidean/rigid motions} $E(\R^d)$ acts smoothly and faithfully.
Therefore the dimension of $\Phi_\p$ is $\dim \Phi_\p \ge \tbinom{d+1}{2}$ and
in case of equality the framework $(G,\p)$ is \emph{infinitesimally rigid}. 

The \emph{rigidity matrix} $R = R(G,\p) \in (\R^d)^{E \times V}$ of $(G,\p)$
is the Jacobian matrix of $\Phi$ evaluated at $\p$. Invoking the Implicit
Function Theorem, it is easy to see that $(G,\p)$ is infinitesimally rigid if
and only if $\rk R = dv - \tbinom{d+1}{2}$.  A \emph{stress} on the framework
$(G,\p)$ is an assignment $\wb = (\w_e)_{e\in E} \in \R^E$ of weights
$\omega_e\in\R$ to the edges $e\in E$ such that there is an equilibrium
$\sum_{u : uv \in E} \w_{uv} (p_v - p_u) = 0$ at  every vertex $v \in V$. We
denote by $S(G,\p) = \{ \wb \in \R^E : \wb R = 0 \}$ the kernel of $R^\top$,
called the \emph{space of stresses} on $(G,\p)$.

\begin{thm}[Whiteley {\cite[Thm.~8.6 with Thm.~2.9]{whi84}}]
\label{thm:poly_rigid}
    Let $P \subset \R^d$ be a $d$-polytope. Let $G = G(P) = (V,E)$ be the
    graph obtained from a triangulation of the $2$-skeleton of $P$ without new
    vertices and let $\p = \p(P)$ be the vertex coordinates. Then the
    resulting framework $(G,\p)$ is infinitesimally rigid.
\end{thm}

The above theorem makes no reference to the triangulation of the $2$-skeleton.
The important fact to note is that the graph $G$ of Theorem
\ref{thm:poly_rigid} will have  exactly $e := |E| = f_1(P) + f_{02}(P) - 3
f_2(P)$ edges: In addition to the $f_1(P)$ edges of $P$, $k-3$ edges are
needed for every $2$-face with $k$ vertices.

For the dimension of the space of stresses $S(G,\p)$ we get
\[
\begin{aligned}
  0\ \ \le\ \ \dim S(G,\p)&\ \ =\ \ e - \rk R\\ 
                          &\ \ =\ \ e - dv + \tbinom{d+1}{2} \\
                          &\ \ =\ \ f_1(P) + f_{02}(P) - 3 f_2(P) - df_0(P) +
                                   \tbinom{d+1}{2} \\
                          &\ \ =\ \  g^\tor_2(P).
\end{aligned}
\]

Now let $P$ be a centrally symmetric $d$-polytope, $d\ge3$. Let $G = G(P) =
(V,E)$ be the graph in Theorem \ref{thm:poly_rigid} obtained from a
triangulation that respects the central symmetry of the $2$-skeleton and let
$\p = \p(P)$ be the vertex coordinates of $P$. The antipodal map $\x \mapsto
-\x$ induces a free action of the group $\mathbb{Z}_2$ on the graph $G$. 
We denote by $\bV = V / \mathbb{Z}_2$ and $\bE = E / \mathbb{Z}_2$ the
respective quotients and, after choosing representatives, we denote by $V =
V^+ \uplus V^-$ and $E = E^+ \uplus E^-$ the decompositions of the set of
vertices and edges according to the action. Since the action is free we have
$|\bV| = |V^\pm| = \tfrac{v}{2}$ and $|\bE| = |E^\pm| = \tfrac{e}{2}$.  

Concerning the rigidity matrix, it is easy to see that
\[
    R\ \ =\ \ 
    \bordermatrix{
                     & \scriptstyle V^+ &\scriptstyle V^- \cr
    \scriptstyle E^+ & \phantom{-}R_1  & \phantom{-}R_2   \cr
    \scriptstyle E^- & -R_2 &-R_1                         \cr
    }
    \ \in\ (\R^d)^{V \times E}
\]
with labels above and to the left of the matrix. The embedding $\p = \p(P)$
respects the central symmetry of $G$ and we can augment the edge function by a
second component that takes the symmetry information into account:
\begin{eqnarray*}
    \Sphi:\quad (\R^d)^{V^+} \times (\R^d)^{V^-} &\rightarrow& \R^E \times
                                                    (\R^d)^{\bV}\\
    \p = (\p_{V^+}, \p_{V^-}) &\mapsto& \left( \Phi(\p), \p_{V^+} + \p_{V^-}
                                                    \right).
\end{eqnarray*}
Thus $\Sphi$ additionally measures the degree of asymmetry of the embedding.
By the symmetry of $P$, $\Sphi(\p) = (\Phi(\p), 0 )$ for $\p=\p(P)$.  The
preimage of this point under $\Sphi$ is $\Sphi_\p \subset \Phi_\p$, the set of
all centrally symmetric embeddings with edge lengths $\Phi(\p)$.  Any small
(close to identity) rigid motion that fixes the origin takes $\p\in\Sphi_\p$
to a \emph{distinct} centrally symmetric realization $\p'\in\Sphi_\p$.  Thus
the action of the subgroup $O( \R^d )$, the group of orthogonal
transformations, on $\Sphi_\p$ locally gives a smooth embedding.  It follows
that~$\dim \Sphi_\p \ge\dim O(\R^d)= \tbinom{d}{2}$ and thus
\begin{equation}\label{eqn:rk_Rcs1}
    \rk R^\sym \ \ \le\ \  dv - \tbinom{d}{2},
\end{equation}
where we can compute the rank of $R^\sym$, the Jacobian of $\Phi^\sym$ at
$\p$, as
\begin{equation}\label{eqn:rk_Rcs2}
    \rk R^\sym\ \ =\ \   
    \rk \begin{pmatrix} 
        \phantom{-}R_1 & \phantom{-}R_2 \\
        -R_2 & -R_1 \\
        I_{V^+} & I_{V^-} \\
    \end{pmatrix}
  \ \ =\ \ \frac{dv}{2} + \rk \left( R_1 - R_2 \right).
\end{equation}

\begin{proof}[Proof of Theorem \ref{thm:g2_symm}]
    Consider the space of \emph{symmetric stresses}, that is, the linear
    subspace 
    \[
    S^\sym(G,\p) = \{ \w = (\w_{E^+},\w_{E^-}) \in S(G,\p) : \w_{E^+} =
    \w_{E^-} \}
        \cong \{ \overline{\w} \in \R^{\bE} : \overline{\w}\left( R_1 - R_2 \right) = 0
        \}.
    \]
    From (\ref{eqn:rk_Rcs1}) and (\ref{eqn:rk_Rcs2}) it follows that
    \[
        \dim S^\sym(G,\p)\ \ =\ \ 
        \frac e2 - \rk(R_1-R_2)
        \ \ \ge\ \ \frac{e}{2} - \frac{dv}{2} + \binom{d}{2}. 
    \]
    The theorem follows from noting that $S^\sym(G,\p) \subseteq S(G,\p)$ and
    therefore 
    \[
    e - dv + \binom{d+1}{2}\ \ \ge\ \ \frac{1}{2}(e - dv) + \binom{d}{2}.
    \]
\vskip-8mm
\end{proof}
\medskip

Theorem \ref{thm:g2_symm} can also be deduced from the following result of
A'Campo-Neuen \cite{aca06}; see also \cite{aca99}.

\begin{thm}[{\cite[Theorem 2]{aca06}}]\label{thm:aca}
    Let $P$ be a centrally symmetric $d$-polytope and let $h^\tor_P(t) =
    \sum_{i=0}^d h^\tor_i(P)\, t^i$ be its toric $h$-polynomial. Then the
    polynomial
    \[
        h^\tor_P(t) - h^\tor_{C_d^\dual}(t) = 
        h^\tor_P(t) - ( 1 + t)^d \in \mathbb{Z}[t]
    \]
    is palindromic and unimodal with non-negative, even coefficients. In
    particular, 
    \[
        g^\tor_i(P) = h^\tor_i(P) - h^\tor_{i-1}(P) \ge \tbinom{d}{i} -
        \tbinom{d}{i-1} \text{ for all } 1 \le i \le \left\lfloor \tfrac{d}{2}
        \right\rfloor.
    \]
\end{thm}

The proof of Theorem \ref{thm:aca} relies on the (heavy) machinery of
combinatorial intersection cohomology for fans. Theorem \ref{thm:g2_symm}
concerns the special case of the coefficient of the quadratic term. In light
of McMullen's \emph{weight algebra} \cite{mcm96}, it would be interesting to
know whether/how Theorem \ref{thm:aca} can be deduced by considering
(generalized) stresses. A connection between the combinatorial intersection
cohomology set-up for fans and rigidity was established by
Braden~\cite[Sect.~2.9]{Braden06}.

\subsection{Large centrally symmetric $\mathbf{4}$-polytopes} 
\label{sec:large_4polys}

In order to prove conjectures \A\ and \B\ for large polytopes, we need one
more ingredient.
\begin{prop} 
    Let $P$ be a $4$-polytope. Then 
    \begin{equation}
        \begin{aligned}
        f_{03}(P) &\ \ \le\ \ 4 f_2(P) - 4 f_3(P)\\
                  &\ \ =  \ \ 4 f_1(P) - 4 f_0(P).\\
        \end{aligned}
        \label{eqn:f03}
    \end{equation}
    Equality holds if and only if $P$ is \emph{center-boolean}, i.e.\ if every
    facet is simple.
\end{prop}
\begin{proof}
    The inequality was first proved by Bayer \cite{bay87}. Every facet $F$ of
    $P$ is a $3$-polytope satisfying $3f_0(F) \leq 2f_1(F)$. By summing up
    over all facets of $P$ we get
    \[ 
        3f_{03}(P) = \sum_{F\text{ facet}} 3f_0(F) \leq \sum_{F\text{ facet}}
        2f_1(F) = 2f_{13}(P) . 
    \]
    By one of the Generalized Dehn-Sommerville Equations \cite{bb85} we have
    \[ 
        f_{03} - f_{13} + f_{23} = 2 f_3 , 
    \]
    which, together with $f_{23} = 2f_2$ immediately implies the asserted
    inequality.  Equality holds if the above inequality for $3$-polytopes
    holds with equality for all facets of $P$, which means that all facets are
    simple $3$-polytopes. The equality in the assertion is Euler's equation.
\end{proof}

Combining the inequalities (\ref{eqn:g2}) and (\ref{eqn:f03}), we obtain
\begin{equation}
    \begin{aligned}
    f_2 &\ \ \ge\ \  \tfrac{1}{4}(3f_0 + 7f_3) -2 \ \ = \ \ f_3 +
                     \tfrac{3}{4}(f_0+f_3) - 2\\ 
    f_1 &\ \ \ge\ \  \tfrac{1}{4}(7f_0 + 3f_3) -2 \ \ = \ \ f_0 +
                     \tfrac{3}{4}(f_0+f_3) - 2.\\
    \end{aligned}
    \label{eqn:01bounds}
\end{equation}
In terms of $f_0$ and $f_3$ this gives
\[
    s(P) \ge \tfrac{14}{4}(f_0 + f_3) - 3 \ge 81
\]
where the last inequality holds if $P$ is large.

To prove conjecture \B\ for large polytopes, we have to show that the
$f$-vector of every large polytope is component-wise larger than the
$f$-vector of one of the following four Hanner polytopes:
\[
    \begin{array}{l|@{\quad(}r@{,\,}r@{,\,}r@{,\,}r@{)}}
                         & f_0 & f_1 & f_2 & f_3 \\
        \hline
        C_4              & 16  & 32  & 24  & 8   \\
        C^\dual_4        &  8  & 24  & 32  & 16  \\
        \bip C_3         & 10  & 28  & 30  & 12  \\
        \prism C^\dual_3 & 12  & 30  & 28  & 10  \\
    \end{array}
\]
It suffices to treat the case $f_0 + f_3 = 24$. Indeed, for $f_0 + f_3 \ge 26$
and $f_3 \ge f_0 \ge 10$ we get from~(\ref{eqn:01bounds}) that
\[
    \begin{aligned}
            f_1 &\ge f_0 + 18 \ge 28 \\
            f_2 &\ge f_3 + 18 \ge 30 \\
    \end{aligned}
\]
and thus $f(\bip C_3)$ is componentwise smaller.

We claim that the same bounds hold for $f_0 + f_3 = 24$. Otherwise, if $f_1\le
26$ or $f_2\le 28$, then by using $(\ref{eqn:f03})$ together with $f_0\ge10$
and $f_3\ge12$ we get in both cases that $f_{03} \le 64$.  In fact, we now get
$f_{03} = 64$ from $(\ref{eqn:g2})$, which tells us that $P$ is \emph{center
boolean}, i.e.\  every facet is simple.  Granted that every facet of $P$ is
simple and has at most $6$ vertices, the possible facet types are the
$3$-simplex $\Delta_3$ and the triangular prism $\prism \Delta_2$.  Using the
assumption that $P$ is not simplicial, there is a facet $F \cong \prism
\Delta_2$. The three quad faces of $F$ give rise to three more prism facets
and, due to the number of vertices, no two of them are antipodes.  For the
same reason, any two prism facets cannot intersect in a triangle face.  In
total, we note that $P$ has exactly eight prism facets and four tetrahedra.
Since every antipodal pair of prism facets give a partition of the vertices,
it follows that every vertex is contained in a simplex and exactly $4$ prism
facets. Therefore, every vertex has degree $\ge 6$ and thus $2f_1 \ge 6 \cdot
12$. By Euler's equation, the same holds for $f_2$.

\subsection{Twisted prisms and the small polytopes}
\label{sec:twist}

The class of small cs $4$-polytopes consists of all cs $4$-polytopes $P$ with
$12 \ge f_3(P) \ge f_0(P) = 10$. Since $P$ is not simplicial, $P$ has a facet
$F$ that has $5 = d+1 = f_0(F)$ vertices, and $P = \conv( F \cup -F)$. In
particular, $F$ is a $3$-polytope with $3+2$ vertices, which does not leave
much diversity in terms of combinatorial types. The facet $F$ is
combinatorially equivalent to 
\begin{compactitem}[$\;\;\blacktriangleright$]
    \item a pyramid over a quadrilateral, or
    \item a bipyramid over a triangle.
\end{compactitem}

\begin{dfn}[Twisted prism] Let $Q \subset \R^{d-1}$ be a $(d-1)$-polytope. The
    centrally symmetric $d$-polytope
    \[
    P = \tprism Q = \conv\left( Q \times \{1\} \ \ \cup \ \  -Q \times
                                \{-1\}\right)
    \subset \R^d
    \]
    is called the \emph{twisted prism} over the base $Q$.
\end{dfn}

The following basic properties of twisted prisms will be of good service.
\begin{prop}\label{prop:tprism}
    Let $Q \subset \R^{d-1}$ be a $(d-1)$-polytope and $\tprism Q$ the
    twisted prism over~$Q$.
    \begin{compactenum}[\rm 1.]
        \item If $\mathsf{T}:\R^{d-1}\rightarrow \R^{d-1}$ is a non-singular
            affine transformation, then $\tprism Q$ and $\tprism \mathsf{T}Q$
            are affinely isomorphic.
        \item If $Q = \pyr Q^\prime$ is a pyramid with base $Q^\prime$, then
            $\tprism Q$ is combinatorially equivalent to $\bip \tprism
            Q^\prime$, a bipyramid over the twisted prism over
            $Q^\prime$.\hfill\qed
    \end{compactenum} 
\end{prop}

The second statement of Proposition \ref{prop:tprism} actually proves the
conjectures \A\ and \B\ for half of the small cs $4$-polytopes: Let $P =
\tprism Q$ and $Q$ a pyramid over a quadrilateral. By the second statement $P$
is combinatorially equivalent to $\bip P^\prime$, where $P^\prime$ is a cs
$3$-polytope. In terms of $f$-polynomials, it is easy to show that for a
bipyramid $f_{\bip Q}(t) = (2+t)f_Q(t)$. 
Thus
\[
    s(P) = f_{\bip P^\prime}(1) = 3 f_{P^\prime}(1) \ge 3^4.
\]

Since $\B$ is true in dimension $3$ there is a $3$-dimensional Hanner polytope
$H$ such that $f_i(P^\prime) \ge f_i(H)$ for $i=0,1,2$. From the above
identity of $f$-polynomials it follows that $f_i(\bip P^\prime) \ge f_i(\bip
H)$ for $1 \le i \le 3$, where  $\bip H = I \oplus H$ is a Hanner polytope.

The next lemma shows that the above class already contains all small
polytopes, which finally settles \A\ and \B\ for dimension 4.

\begin{lem} 
    Let $d \ge 4$ and let $P = \tprism F \subset \R^d$ be a cs $d$-polytope with
    $F$ combinatorially equivalent to $\Delta_i \oplus \Delta_{d-i-1}$ and $
    1 \le i \le \tfrac{d-1}{2}$. Then 
    \[
        f_{d-1}(P)\ \ \ge\ \ 2 ( 1  + (i+1)(d-i) )\ \ \ge\ \ 2 ( 2d-1 ).
    \]
\end{lem}
\begin{proof}
    The facet $F$ in $P$ has $(i+1)(d-i)$ ridges and thus $F$ and its
    neighbors account for
    $1 + (i + 1)(d-i)$ facets. The result now follows by considering $-F$ as
    soon as we have checked that no facet $G$ shares a ridge with $F$ and with
    $-F$.  This, however, is impossible, since $G$ would have to have two
    vertex disjoint $(d-2)$-simplices as maximal faces and, therefore, at
    least $f_0(G) \ge 2d-2$ vertices. Thus $2d+2 = f_0(P) \ge f_0(G) + f_0(-G)
    \ge 4d-4$.
\end{proof}

\begin{cor} If $P = \tprism Q$ with $Q \cong \bip \Delta_2$, then $P$ is
    large.
\end{cor}

\section{Conjecture C in dimension 4}
\label{sec:C_dim4}

We will refute conjecture \C\ \emph{strongly} for dimension $4$: We exhibit a
flag-functional $\a \in \FF_4$ and a cs $4$-polytope $P$ such that $\a(P) <
\a(H)$ for every $4$-dimensional Hanner polytope $H$. 

Geometrically, this means that there is an oriented hyperplane in the vector
space $\R^{2^{[d]}}$ that has the flag vector $(f_S(P))_S$ on its negative
side, but all the flag-vectors of Hanner polytopes on its positive side, while
some parallel hyperplane has the flag-vectors of \emph{all} (not-necessarily
cs) $4$-polytopes on its positive side.

For this, consider the two functionals
\begin{eqnarray*}
        \ell_1(P) &=& f_{02}(P) - 3 f_2(P) \\
        \ell_2(P) &=& f_{13}(P) - 3 f_1(P) \\
                  &=& f_{02}(P) - 3 f_1(P). 
\end{eqnarray*}
Let $F_k(P)$ be the number of $2$-faces with exactly $k$ vertices. Then
$f_{02}(P) = \sum_{k\ge 3} k\cdot F_k(P)$. Thus $\ell_1(P) = \sum_{k\ge 4}
(k-3) \cdot F_k(P)$, which is clearly non-negative for every $4$-polytope. In
case of equality the polytope is \emph{$2$-simplicial}. For the second
functional note that $\ell_2(P) = \ell_1(P^\dual) \ge 0$ and the bound is
attained by  the \emph{$2$-simple} polytopes. Thus, the functional 
\[\a(P)\ \  :=\ \  \frac{1}{2}(\ell_1 + \ell_2)\ \ =\ \ f_{02} -
\frac{3}{2}(f_1 + f_2)
\]
is non-negative for all $4$-polytopes; it vanishes exactly for $2$-simple
$2$-simplicial polytopes. (See~\cite{paf04} for examples of such polytopes.)

Consider the cs $4$-polytope
\[
    P_4\ \ :=\ \ [-1,+1]^4 \cap \{ \x \in \R^4 : -2 \le x_1 + \cdots + x_4 \le 2 \}
\]
which arises from the $4$-cube $C_4$ by \emph{chopping off} the vertices
$\pm\1$ by hyperplanes that pass through the respective neighbors.  It is
straightforward to verify that the $f$-vector of $P_4$ is
\[
    f(P_4) = (10,32,36,14).
\]
Indeed, the only faces that go missing are the $2\cdot4$ edges incident to the
two vertices; the added faces are the faces of strictly positive dimension of
the vertex figures at $\1$ and $-\1$. Concerning the number of
vertex--$2$-face incidences: there are only triangles and quadrilaterals. The
number of triangles is twice the number of $2$-faces and facets incident to
any given vertex. Thus, $f_{02} = 3 \cdot 20 + 4 \cdot 12 = 108$ and $\a(P_4) =
6$.

Theorem \ref{thm:Cdim4} now follows from inspecting the
following table, which lists in its first row the data for $P_4$, and
then (extended) data for the $4$-dimensional Hanner polytopes:
\[
    \begin{array}{l|@{\quad(} r @{,\,} r @{,\,} r @{\,} r @{\,)\quad} r @{\quad} r}
                         & f_0 & f_1 & f_2 & f_3 & f_{02} & \a\\
        \hline
        P_4                & 10  & 32  & 36  & 14  & 108& 6\\[2mm]
        C_4              & 16  & 32  & 24  & 8   & 96 & 12 \\
        C^\dual_4        &  8  & 24  & 32  & 16  & 96 & 12 \\
        \bip C_3         & 10  & 28  & 30  & 12  & 96 & 9 \\
        \prism C^\dual_3 & 12  & 30  & 28  & 10  & 96 & 9 \\
    \end{array}
\]

\section{The central hypersimplices $\HS_k=\Delta(k,2k)$}
\label{sec:hypsimp}

For natural numbers $d > k > 0$, the \emph{$(k,d)$-hypersimplex} 
is the $(d-1)$-dimensional polytope
\[
    \Delta(k,d) = \conv \left\{ \x \in \ZO^d :  x_1 + x_2 + \cdots + x_d = k
    \right\} \subset \R^d.
\]
Hypersimplices were considered as (regular) polytopes in \cite[\S 11.8]{cox73}
(see also \cite[Sect.~3.3.2]{paf04} and \cite[Exercise 4.8.16]{grue03}),  as
well as in connection with algebraic geometry in \cite{gkz94}, \cite{gel82},
and \cite{stu96}. 

One rather simple observation is that $\Delta(k,d)$ and $\Delta(d-k,d)$ are
affinely isomorphic under the map $\x \mapsto \1 - \x$.  In particular, the
hypersimplex $\HS_k := \Delta(k,2k)$ is a centrally symmetric
$(2k-1)$-polytope with $f_0(\HS_k) = \binom{2k}{k}$ vertices.

In a different, full-dimensional realization, the central hypersimplex
is given by
\[
    \HS_k\ \ \cong\ \ \conv \left\{ \x \in \PM^{2k-1} :
     -1\le x_1 + x_2 + \cdots + x_{2k-1} \le  1 \right\}.
\]
From this realization it is easy to see that for $k \ge 2$ the hypersimplex
$\HS_k$ is a twisted prism over $\Delta(k,2k-1)$ with $f_{2k-2}(\HS_k) = 4k =
2(2k-1) + 2$ facets: Since the above realization lives in an odd-dimensional
space, the sum of the coordinates for any vertex is either $+1$ or $-1$. The
points satisfying $\sum_ix_i = 1$ form a face that is affinely isomorphic to
$\Delta(k,2k-1)$. To verify the number of facets, observe that $\HS_k$ is the
intersection of the $2k$-cube with a hyperplane that cuts all its $4k$ facets.

We will show that in  odd dimensions $d = 2k-1 \ge 5$ a $d$-dimensional Hanner
polytope that has no more facets than $\HS_k$ has way too many vertices for
conjecture \B. In even dimensions $d \ge 6$ Theorem \ref{thm:dim5} follows
then by taking a prism over $\HS_k$. The following proposition gathers the
information needed  about Hanner polytopes.
\begin{prop}\label{prop:Hanner2d}
    Let $H$ be a $d$-dimensional Hanner polytope. Then 
    \begin{compactenum}[\rm(a)]
        \item $f_{d-1}(H) \ge 2d$.
        \item If $f_{d-1}(H) = 2d$, then $H$ is a $d$-cube.
        \item If $f_{d-1}(H) = 2d+2$, then $H = C_{d-3} \times C_3^\dual$.
    \end{compactenum}
\end{prop}
\begin{proof}
    Since all three claims are certainly true for Hanner polytopes of
    dimension $d \le 3$, let us assume that $d \ge 4$.  By definition, $H$ is
    the direct sum or product of two Hanner polytopes $\Hp$ and $\Hpp$ of
    dimensions $i$ and $d-i$ with $1 \le i \le \tfrac{d}{2}$. 

    If $H = \Hp \oplus \Hpp$, then, by induction on $d$, we get
    \[
        f_{d-1}(H) = f_{i-1}(\Hp) \cdot f_{d-i-1}(\Hpp) \ge 4i(d-i) \ge 2d+4.
    \]
    Therefore, we can assume that $H = \Hp \times \Hpp$ and $f_{d-1}(H) =
    f_{i-1}(\Hp) + f_{d-i-1}(\Hpp) \ge 2d$ which proves (a). The condition in
    (b) is satisfied if and only if it is satisfied for each of the two
    factors. Therefore, by induction, both factors are cubes and so is their
    product.

    Similarly, the condition in (c) is satisfied iff it is satified for one of
    the two factors. By using (a) we see that the remaining factor is a cube,
    which proves (c).
\end{proof}

\begin{proof}[Proof of Theorem \ref{thm:dim5}]
Let $d = 2k - 1 \ge 5$ and let $H$ be a $d$-dimensional Hanner polytope with
$f_i(H) \le f_i(\HS_k)$ for all $i = 0, \dots, d-1$. Since the hypersimplex
$\HS_k$ has $2d+2$ facets, it follows from Proposition \ref{prop:Hanner2d}
that $H$ is either $C_{2k-1}$ or $C_{2k-4} \times C_3^\dual$. In either
case, the Hanner polytope satisfies $f_0(H) \ge 3 \cdot 2^{2k-3} >
\tbinom{2k}{k}$ , where the last inequality holds for $k \ge 3$.

For even dimensions $d = 2k$ consider $\prism \HS_k = I \times \HS_k$, which
has $ 2(2k-1) + 4 = 2d+2$ facets. Again by Proposition \ref{prop:Hanner2d}, a
Hanner polytope $H$ with componentwise smaller $f$-vector is of the form $I
\times H^\prime$ and the result follows from the odd case.
\end{proof}

\section{Two more examples}
\label{sec:examples}

We wish to discuss two examples of centrally symmetric polytopes that exhibit
some remarkable properties, two of which are being \emph{self-dual} and being
counter-examples to conjecture \C. Both polytopes are instances of
\emph{Hansen polytopes} \cite{han77}, for which we sketch the construction.

Let $G = (V,E)$ be a \emph{perfect} graph on the vertex set $V = \{1, \dots,
d-1\}$, that is, a simple, undirected graph without induced odd cycles of
length $\ge 5$ (cf.\ Schrijver \cite[Chap.~65]{Schrijver:VolB}).  Let $\ic{G}
\subseteq 2^V$ be the \emph{independence complex} of $G$. So $\ic{G}$ is the
simplicial complex on the vertices $V$ defined by the relation that $S
\subseteq V$ is contained in $\ic{G}$ if and only if the vertex induced
subgraph $G[S]$ has no edges. To every independent set $S \in \ic{G}$
associate the (characteristic) vector $\tchi_S \in \PM^{d-1}$ with
$(\tchi_S)_i = +1$ if and only if $i \in S$. The collection of vectors is a
subset of the vertex set of the $(d-1)$-cube. Let $P_{\ic{G}} = \conv \{
\tchi_S : S \in \ic{G} \} \subset [-1,+1]^{d-1}$ be the vertex induced
subpolytope. The \emph{Hansen polytope} $H(G)$ associated to $G$ is the
twisted prism over $P_{\ic{G}}$. In particular, $H(G)$ is a centrally
symmetric $d$-polytope with $f_0(H(G))  = 2\,|\ic{G}|$ vertices.  A graph $G =
(V,E)$ is \emph{self-complementary} if $G$ is isomorphic to its complementary
graph $\overline{G} = (V, \tbinom{V}{2}{\setminus}E)$.

\begin{prop} 
    If $G = (V,E)$ is a self-complementary, perfect graph on $d-1$ vertices,
    then $H(G)$ is a centrally symmetric, self-dual $d$-polytope.
\end{prop}

\begin{proof}
    By \cite[Thm.~4]{han77}, the polytope $H(G)^\dual$ is isomorphic to
    $H(\overline{G}) = H(G)$.
\end{proof}

\begin{example}\label{G_4}
Let $G_4$ the path on four vertices $v_1,v_2,v_3,v_4$.  This is a
self-complementary perfect graph, so $H(G_4)$ is a $5$-dimensional self-dual
cs polytope.  We compute its $f$-vector, and compare it to the $f$-vectors of
the $5$-dimensional hypersimplex $\HS_3$ and of the eight $5$-dimensional
Hanner polytopes. This results in the following table (the four Hanner
polytopes not listed are the duals of the ones given here, with the
corresponding reversed $f$-vectors):
\[
  \begin{array}{l@{\qquad}|@{\quad(\,} r @{,\,} r @{,\,} r @{,\,} r @{,\,}  r @{\,)\quad} c@{\quad} c}
                      & f_0 & f_1 & f_2 & f_3 & f_4 & f_0+f_4 & s \\
  \hline
  H(G_4) \hspace{8mm} & 16  & 64  & 98  & 64  & 16  & 32 & 259 \\[2mm]
  \HS_3 \hspace{12mm} & 20  & 90  & 120 & 60  & 12  & 32 & 303 \\[2mm]
  C_5^\dual            & 10  & 40  & 80  & 80  & 32  & 42 & 243 \\
  \bip\bip C_3         & 12  & 48  & 86  & 72  & 24  & 36 & 243 \\
  \bip\prism C_3^\dual & 14  & 54  & 88  & 66  & 20  & 34 & 243 \\
  \prism C^\dual_4     & 16  & 56  & 88  & 64  & 18  & 34 & 243
\end{array}
\]
Thus $H(G_4)$ refutes conjecture~\B\ in dimension~$5$ \emph{strongly}: its
value for $f_0+f_4$ is smaller than for any Hanner polytope.  Furthermore,
$H(G_4)$ has a smaller face number sum $s$ than the hypersimplex, so in
that sense it is even a better example to look at in view of conjecture~\A.
\end{example}

\begin{example}
  Let $G_5$ be the path on five vertices $v_1,v_2,v_3,v_4,v_5$ (in this
  order), with an additional edge connecting the second vertex $v_2$ to the
  fourth vertex $v_4$ on the path.  This is a self-complementary perfect
  graph, so we obtain a $6$-dimensional self-dual cs polytope $H(G_5)$.  Again
  its $f$-vector can be computed and compared to those of the prism over the
  $5$-dimensional hypersimplex, $I \times \HS_3$, which we had used for
  Theorem~\ref{thm:dim5} as well as the eighteen Hanner polytopes in dimension
  $6$ (again we do not list the duals explicitly):

\[
  \begin{array}{l@{\qquad}|@{\quad(\,} r @{,\,} r @{,\,} r @{,\,} r @{,\,} r @{,\,}  r  @{\,)\quad} c @{\quad} c}
                            & f_0 & f_1 & f_2 & f_3 & f_4 & f_5 & f_0+f_5 & s \\
  \hline
  H(G_5) \hspace{8mm}       & 24 &  116 &  232 &  232 &  116 &  24 & 48 & 745 \\[2mm]
  \prism\HS_3 \hspace{12mm} & 40  & 200 & 330 & 240 & 84  & 14  & 54 & 908 \\[2mm]
  C_6^\dual                 & 12  & 60  & 160 & 240 & 192 & 64  & 76  & 729 \\
  \bip\bip\bip C_3          & 14  & 72  & 182 & 244 & 168 & 48  & 62  & 729 \\
  \bip\bip\prism C_3^\dual  & 16  & 82  & 196 & 242 & 152 & 40  & 56  & 729 \\
  \bip\prism C_4^\dual      & 18  & 88  & 200 & 240 & 146 & 36  & 54  & 729 \\
  \bip\bip C_4              & 20  & 100 & 216 & 232 & 128 & 32  & 52  & 729 \\
  \prism C_5^\dual          & 20  & 90  & 200 & 240 & 144 & 34  & 54  & 729 \\
  \bip\prism\bip C_3        & 22  & 106 & 220 & 230 & 122 & 28  & 50  & 729 \\
  \prism\bip\bip C_3        & 24  & 108 & 220 & 230 & 120 & 26  & 50  & 729 \\
  C_3 \oplus C_3            & 16  & 88  & 204 & 240 & 144 & 36  & 52  & 729 
\end{array}
\]
  Thus $H(G_5)$ is a self-dual cs polytope that also refutes conjecture~\B\ in
  dimension~$6$ \emph{strongly}.  Moreover, also looking at the pair
  $(f_1,f_4)$ suffices to derive a contradiction to conjecture~\B. In these
  respects, $H(G_5)$ is the nicest and strongest counter-example that we
  currently have for conjecture~\B\ in dimension~$6$.
\end{example}

Note that there are no self-complementary (perfect) graphs on $6$ or on $7$
vertices, since $\binom62=15$ and  $\binom72=21$ are~odd.  Thus, we cannot
derive self-dual polytopes in dimensions $7$ or $8$ from Hansen's
construction.

The Hansen polytopes, derived from perfect graphs, are subject to further
research. For example, $H(G_4)$ and $H(G_5)$ are interesting examples in view
of the Mahler conjecture, since they exhibit only a small deviation from the
Mahler volume of the $d$-cube, which is conjectured to be minimal (see
Kuperberg \cite{Kuperberg:Mahler2} and Tao~\cite{tao:_open}).

The Hansen polytopes in turn are special cases of \emph{weak Hanner
polytopes}, as defined by Hansen \cite{han77}, which are twisted prisms over
any of their facets.  Greg Kuperberg has observed that all of these are
equivalent to $\pm1$-polytopes.

\bibliographystyle{siam}
\begin{small}
    \bibliography{3dconj}
\end{small}

\end{document}